\documentclass[10pt]{amsart}
\usepackage{tikz-cd}
\usepackage{fancyhdr}
\usepackage{latexsym, amssymb, amscd, amsfonts,pstricks,enumitem,amsmath, young}
\usepackage{latexsym, amssymb, amscd, amsfonts,pstricks,enumitem,amsmath}
\usepackage{amsmath,amsthm,amssymb,mathrsfs}
\usepackage{verbatim}

\usepackage{ifthen}
\usepackage{longtable}

\theoremstyle{plain}
\newtheorem{theorem}{Theorem}[section]

\newtheorem{proposition}[theorem]{Proposition}

\newtheorem{conjecture}[theorem]{Conjecture}
\newtheorem{cor}[theorem]{Corollary}
\newtheorem{def-thm}[theorem]{Definition-Theorem}
\newtheorem{lemma}[theorem]{Lemma}

\newtheorem{defi}[theorem]{Definition}

\newtheorem*{claim}{Claim}

\DeclareMathOperator{\codim}{codim}
\DeclareMathOperator{\supp}{Supp\hspace{1pt}}

\theoremstyle{definition}

\newtheorem{remark}[theorem]{Remark}

\def\min{\mathop{\mathrm{min}}}

\def\PP{\mathbb P}

\let\e\epsilon

\let\l\lambda

\begin{document}
\title{{Greatest common divisors of integral points of numerically equivalent divisors}}
\author{Julie Tzu-Yueh Wang}
\address{Institute of Mathematics, Academia Sinica \newline
\indent No.\ 1, Sec.\ 4, Roosevelt Road\newline
\indent Taipei 10617, Taiwan}
\email{jwang@math.sinica.edu.tw}
\author{Yu Yasufuku}
\address{Department of Mathematics, College of Science and Technology, Nihon University\newline
\indent 1-8-14 Kanda-Surugadai, Chiyoda-ku\newline
\indent 101-8308 Tokyo, Japan\newline}
\email{yasufuku@math.cst.nihon-u.ac.jp}
\begin{abstract}
We generalize the G.C.D. results of Corvaja--Zannier and Levin on $\mathbb G_m^n$ to more general settings. More specifically, we analyze the height of a closed subscheme of codimension at least $2$ inside an $n$-dimensional Cohen-Macaulay projective variety, and show that this height is small when evaluated at integral points with respect to a divisor $D$ when $D$ is a sum of $n+1$ effective divisors which are all numerically equivalent to some multiples of a fixed ample divisor. Our method is inspired by Silverman's G.C.D. estimate as an application of Vojta's conjecture, which is substituted by  a more general version of Schmidt's subspace theorem of Ru--Vojta in our proof.
 \end{abstract}
 \thanks{2010\ {\it Mathematics Subject Classification.}
11J97, 11J87, 14G05, 32A22.}
\thanks{The first named author was supported in part by Taiwan's MoST grant 108-2115-M-001-001-MY2, and the second named author was supported in part by Japan's JSPS grant 15K17522 and 19K03412.}

\maketitle \pagestyle{myheadings}
\markboth{Julie Tzu-Yueh Wang and Yu Yasufuku}{G.C.D. of integral points of numerically equivalent divisors}


\section{Introduction and Statements}\label{introduction}
\setcounter{equation}{0}
In a recent work \cite{levin_gcd}, Levin obtained the following result which bounds the greatest common divisor of multivariable polynomials evaluated at $S$-unit arguments.  This result is a generalization of results of Bugeaud-Corvaja-Zannier \cite{Bugeaud:Corvaja:Zannier:2003}, Hern\'andea-Luca \cite{Hernandez:Luca:2003} and Corvaja-Zannier \cite{Corvaja:Zannier:2003}, \cite{Corvaja:Zannier:2005}.  We refer to \cite{levin_gcd} for a survey of these related results.
\begin{theorem}[{\cite[Theorem 1.1]{levin_gcd}}]\label{Levin:thm1.1}
Let $\Gamma \subseteq \mathbb G^r_m(\overline{\mathbb Q})$ be a finitely generated group and fix nonconstant coprime polynomials $f(x_1,\dots,x_r), g(x_1,\dots, x_r) \in \overline{\mathbb Q}[x_1,\dots,x_r]$ which do not both vanish at the origin $(0,\dots,0)$.  Then, for each $\epsilon > 0$, there exists a finite union $Z$ of translates of proper algebraic subgroups of $ \mathbb G_m^r$ so that
$$
\log \gcd (f(\mathbf{u}), g(\mathbf{u}) ) < \epsilon \max_i \{h(u_i)\}
$$
for all $\mathbf{u} = (u_1,\dots,u_r) \in \Gamma \setminus Z$.
\end{theorem}
The greatest common divisor on the left-hand side of the above inequality is a generalized notion of the usual quantity for integers, adapted to algebraic numbers to also account for the archimedean contributions \cite[Definition 1.4]{levin_gcd}.   He then further elaborated Theorem \ref{Levin:thm1.1} into the following geometric setting in projective spaces. Here, $h_Y$ is a height associated to the closed subscheme $Y$ (see Section \ref{Preliminary}).

 \begin{theorem}\label{LevinTheorem 1.16} \cite[Theorem 1.16]{levin_gcd}
  Let $Y$ be a  closed subscheme of $\mathbb P^n$ of codimension at least 2, defined over a number field.
  Suppose that
 $$
 P_0:=[1:0\cdots:0],\hdots,P_n:=[0:0\cdots:0:1]\notin \supp Y.
 $$
 Let $\Gamma \subseteq \mathbb G^n_m(\overline{\mathbb Q})$ be a finitely generated group.     Then for all $\epsilon>0$, there exists a finite union $Z$ of translates of proper algebraic subgroups of $\mathbb G^n_m$ such that
 \begin{align}
h_{Y}(P)\le   \epsilon h(P) +O(1),
\end{align}
for all $P\in  \Gamma\setminus Z\subset \mathbb P^n(\overline{\mathbb Q})$.
\end{theorem}

The following is our main result that is a direct generalization of Theorem \ref{LevinTheorem 1.16}.
The proof of Theorems \ref{Levin:thm1.1} and \ref{LevinTheorem 1.16} is a sophisticated, but a direct application of the classical Schimidt's subspace theorem.
Our proof of Theorem \ref{heightEnumerical} below is based on  an idea of Silverman in \cite{Silverman2005} and the recent work \cite{ruvojta} of Ru and Vojta.
Therefore, it  provides a different proof for
Theorem \ref{LevinTheorem 1.16}.

\begin{theorem}\label{heightEnumerical}
 Let $V$ be a
  Cohen--Macaulay projective variety   of dimension $n$  defined over a number field $k$, and let $S$ be a finite set of places of $k$.   Let $D_1, \dots, D_{n+1} $ be    effective  Cartier divisors of $V$  defined over $k$  in general position.   Suppose that there exist an ample Cartier divisor $A$ on $V$ and positive integers $d_i$ such that $D_i\equiv d_iA$ for all $1\le i\le n+1$.  Let $Y$ be a  closed subscheme of $V$ of codimension at least 2  that does not contain any  point  of the set
\begin{equation}\label{eq:intersection}
\bigcup_{i=1}^{n+1} \left( \bigcap_{1\le j\ne i\le n+1}\supp D_j\right).
\end{equation}
 Let $\epsilon>0$.  Then there exists a proper Zariski closed set $Z$ of $V(k)$ such that for any set $R$ of $(\sum_{i=1}^{n+1} D_i,S)$-integral points on $V$ we have
  \begin{align}\label{eq:height}
h_Y(P)\le   \epsilon h_A(P)+O(1)
\end{align}
for all for all $P\in R \setminus Z$.
\end{theorem}

Here, $\equiv$ denotes numerical equivalence of divisors, and the notion of $(D,S)$-integral points will be discussed in more detail in Section \ref{Preliminary}.

Given a finitely generated subgroup $\Gamma \subseteq \mathbb G^n_m(\overline{\mathbb Q})$ and a closed subscheme $Y\subset\mathbb P^n$ defined over a number field, there exist a number field $k$ and a finite set of places $S$ of $k$ (containing the archimedean places) such that $\Gamma \subseteq \mathbb G^n_m(\mathcal O_S)$ and $Y$ is defined over $k$.  Therefore, we may replace $\Gamma$ in Theorem \ref{LevinTheorem 1.16}  by $\mathbb G^n_m(\mathcal O_S)$ and assume that $Y$ is defined over $k$.
Let $D_1:=\{x_0=0\},\hdots,D_{n+1}:=\{x_n=0\}$ be the divisors given by the coordinate hyperplanes.  Then the set \eqref{eq:intersection} is exactly the set of points $\{P_0,\hdots,P_n\}$ in Theorem \ref{LevinTheorem 1.16}.  Moreover, the set $\mathbb G_m^n(\mathcal O_S)$ is a set of $(\sum_{i=1}^{n+1} D_i,S)$-integral points on $\mathbb P^n(k)$.  Since the Zariski closure of a subset of $\Gamma$ is a finite union of translates of algebraic subgroups by Laurent \cite{laurent}, Theorem \ref{LevinTheorem 1.16} is a direct consequence of Theorem \ref{heightEnumerical}.  With a bit more work, the results of \cite{yasuams} can also be seen as special cases of Theorem \ref{heightEnumerical}.

We note that in \cite{levin_gcd}, Levin divides his argument into the ``counting function part'' $N_{Y,S}$ and the ``proximity function part'' $m_{Y,S}$ (see Section \ref{Preliminary} for the definitions of $N_{Y,S}$ and $m_{Y,S}$).  The $N_{Y,S}$ part is the main work (see \cite[Theorem 4.3]{levin_gcd}), and for a closed subscheme $Y$ of $\mathbb P^n$ of codimension at least 2, he shows that
\begin{align}\label{eq:aaronscounting}
N_{Y,S}(P)\le   \epsilon h(P) +O(1)
\end{align}
for all $P\in  \mathbb G_m^n(\mathcal O_S)$ outside a Zariski-closed proper subset.
Our main argument needs to consider $h_Y$, i.e. combining the $N_{Y,S}$ part and the $m_{Y,S}$ part together.  Therefore, we recover \eqref{eq:aaronscounting} with the additional assumption that $Y$ does not intersect the set \eqref{eq:intersection}.
However, when $Y$ contains only points, our method also works without this assumption.  In particular, it recovers a result of Corvaja-Zannier for $n=2$ \cite[Proposition 4]{Corvaja:Zannier:2005}, which is a consequence of the following.

\begin{theorem}\label{dim2}
 Let $V$ be a
 Cohen--Macaulay projective variety   of dimension $n$  defined over a number field $k$, and let $S$ be a finite set of places of $k$.   Let $D_1, \dots, D_{n} $ be    effective  Cartier divisor of $V$  defined over $k$  and in general position.   Suppose that there exists an ample Cartier divisor $A$ on $V$ and positive integer $d_i$ such that $D_i\equiv d_iA$ for all $1\le i\le n$.  Let $Y$ be a  closed subscheme of $V$ of dimension 0.
 Let $\epsilon>0$.  Then there exists a proper Zariski closed set $Z$ of $V(k)$ such that for any set $R$ of $(\sum_{i=1}^{n} D_i,S)$-integral points on $V$ we have
 \begin{align}\label{eq:countingpart}
N_{Y,S}(P)\le   \epsilon h_A(P)+O(1)
\end{align}
for all for all $P\in R \setminus Z$.
\end{theorem}
\begin{cor}[Corvaja-Zannier \cite{Corvaja:Zannier:2005}]\label{CorvajaZannier}
Let $k$ be a number field and $S$ a finite set of places of $k$ containing the archimedean places.  Let $f(X,Y), g(X,Y)\in k[X,Y]$ be nonconstant coprime polynomials.  For all $\epsilon>0$, there exists a finite union $Z$ of translates of proper algebraic subgroups of $\mathbb G_m^2$ such that
$$
-\sum_{v\in M_k\setminus S}\log^-\max\{|f(x,y)|_v, |g(x,y)|_v\}<\epsilon\max\{h(x),h(y)\}
$$
for all $(x,y)\in \mathbb G_m^2(\mathcal O_S)\setminus Z$.
\end{cor}
We note that Corvaja-Zannier also showed that  the translate of one dimensional subgroups of $\mathbb G_m^2$ in $Z$ as stated in Corollary \ref{CorvajaZannier} can be determined effectively.

To further illustrate the ideas of our work, let us first recall the following conjecture of Vojta. (See \cite[Chapter 15]{Vojta}).
\begin{conjecture}[Vojta's Main Conjecture]
Let $k$ be a number field, let $S$ be a finite set of places of $k$.  Let $X$ be a smooth projective variety over $k$.  Let $A$ be an ample divisor on $X$, let $D$ be a simple normal crossing divisor on $X$ and let $K_X$ be a canonical divisor on $X$.
Then for every $\epsilon>0$ there exists a proper Zariski closed subset $Z=Z(X,D,A,\epsilon)$ and a constant $C_{\epsilon}(X,D,A,\epsilon)\in\mathbb R$ the inequality
\begin{align}\label{eqVojtaBound}
  m_{D,S}( P)+h_{K_X}(P)\le \epsilon  h_{A}(P)+C_{\epsilon}
\end{align}
holds for all $P\in  X(k)\setminus Z$.
\end{conjecture}
In \cite[Theorem 6]{Silverman2005}, Silverman considered Vojta's conjecture on the blowup $\pi:  \widetilde X\to X$ of $X$ along $Y$, where $X$ is a smooth variety over $k$, $Y\subset X$ is a smooth subvariety of codimension $r\ge 2$ under the assumption that $-K_X$ is a simple normal crossings divisor and $Y$ intersects the support of $K_X$ transversally (see also \cite{mck,yasumona,yasuforum}).  The key point of our method is to replace Vojta's conjecture on $ \widetilde X$ by a more general version of Schmidt's subspace theorem obtained recently by Ru and Vojta in \cite{ruvojta}.

To illustrate the idea, we consider the particular case when $X=\mathbb P^n$, $-K_X= \sum_{i=0}^n H_i$, where $H_i=\{X_i=0\}$ is the divisor of the coordinate hyperplane. Applying Vojta's conjecture with $D=-\pi^*K_X= \sum_{i=0}^{n} \pi^* H_i$   and ample divisor $\pi^*A-\frac 1{\ell} E$ on $\widetilde X$, where $A=H_0$ and $\ell$ is a   sufficiently  large integer,
the equation \eqref{eqVojtaBound} becomes
\begin{align}\label{VojtaconjectureP}
\sum_{i=0}^n m_{\pi^* H_i,S}( Q) +(r-1+  \frac \epsilon\ell)h_E(Q) \le  (n+1+\epsilon) h_{\pi^* A}(Q) +C_{\epsilon}.
\end{align}

In contrast,  our computation in this situation (see equations \eqref{eq:fromruvojta} and \eqref{eq:replaceD}) shows that the result of \cite{ruvojta} (see Theorem \ref{Ru-Vojta}) implies that for a sufficiently large $\ell$, there exists a proper closed subset $ \widetilde Z\subset  \widetilde X$ such that
\begin{align}\label{eqVojtaBoundintro}
\sum_{i=0}^n m_{\pi^* H_i,S}( Q)+   \frac 1\ell h_E(Q) \le   (n+1 +O(\frac 1{\ell \sqrt \ell}))  h_{\pi^* A}(Q) + C
\end{align}
for all $Q\in    \widetilde X(k)\setminus  \widetilde Z$ provided the supports of $\pi^* H_0,\hdots,\pi^* H_{n}$ are in general position.   We note that this assumption is weaker than what is required for \cite{Silverman2005} since for
example $Y$ does not have to be nonsingular at the intersection points with $H_i$.  While the inequality \eqref{eqVojtaBoundintro} is weaker than \eqref{VojtaconjectureP} deriving from Vojta's conjecture (the coefficient for $h_E$ is smaller),
the estimate \eqref{eqVojtaBoundintro} is still sufficient for our goal of analyzing integral points.  To see this, by functoriality and using the assumption that $P =\pi(Q)$ is  $(\sum_{i=0}^{n} H_i,S)$-integral,
\begin{align}\label{eqVojtaBoundintroY}
  \sum_{i=0}^n m_{\pi^* H_i,S}( Q) = \sum_{i=0}^n m_{H_i,S}( P) +O(1) =(n+1)h_{A}(P)+O(1).
\end{align}
So it follows from \eqref{eqVojtaBoundintro} that
\begin{align}\label{eqVojtaBoundintroY2}
h_Y(P) = h_E(Q) \le   O(\frac 1{\sqrt \ell})  h_{A}(P)+O(1)
\end{align}
for all $P$ in a set of $(\sum_{i=0}^{n} H_i,S)$-integral points outside $\pi( \widetilde Z)$ of $ \mathbb P^n(k)$.
We also note that the idea of replacing Vojta's conjecture  by  the result of Ru and Vojta in \cite{ruvojta} along the argument of Silverman \cite{Silverman2005} first appeared in \cite{GuoWang},    where they apply the complex part of   Ru and Vojta in \cite{ruvojta} for the blowup $ \widetilde X$ of $X=\mathbb P^1\times \mathbb P^1$ along $Y=(1,1)$ and the ample divisor $ \widetilde A:=-\pi^* K_X-E=\pi^*(\{0\}\times\mathbb P^1)+\pi^*(\{\infty\}\times\mathbb P^1)+\pi^*(\mathbb P^1\times\{0\})+\pi^*(\mathbb P^1\times\{\infty\})-E$.(See   \cite[Corollary 1.11]{ruvojta}   for further discussion in the arithmetic side.)

Finally, we discuss the analogous topics and results in Nevanlinna theory.
 A starting point for such results comes from the study of holomorphic curves in semi-abelian varieties by Noguchi, Winkelmann, and Yamanoi, who proved the following:

\begin{theorem}[Noguchi, Winkelmann, Yamanoi {\cite[Th.~5.1]{NWY} (see also \cite[Section 6.5]{NWbook})}]
\label{tNWY}
Let $f:\mathbb{C}\to A$ be a holomorphic map to a semi-abelian variety $A$ with Zariski-dense image. Let $Y$ be a closed subscheme of $A$ with $\codim Y\geq 2$ and let $\epsilon>0$.
\begin{enumerate}
\item[{\rm (i)}]  Then
\begin{align*}
N_{f}(Y,r)\leq_{\rm exc} \epsilon T_f(r).
\end{align*}
\item[{\rm (ii)}] There exists a compactification $\overline{A}$ of $A$, independent of $\epsilon$, such that for the Zariski closure $\overline{Y}$ of $Y$ in $\overline{A}$,
\begin{align*}
T_{\overline{Y},f}(r)\leq_{\rm exc} \epsilon T_f(r).
\end{align*}
\end{enumerate}
\end{theorem}
Here $N_f(Y,r)$ is a counting function associated to $f$ and $Y$, $T_{\overline{Y},f}(r)$ is a Nevanlinna characteristic (or height) function associated to $f$ and $\overline{Y}$, and $T_f(r)$ is any characteristic function associated to an appropriate ample line bundle (see \cite[Section 2]{levinwang} for the relevant definitions from Nevanlinna theory and \cite{NWY} for more discussion).  The notation $ \le_{\rm exc}$ means that the estimate holds for all $r$ outside a set of finite Lebesgue measure, possibly depending on $\epsilon$.
More generally, Noguchi, Winkelmann, and Yamanoi proved a result for $k$-jet lifts of holomorphic maps to semi-abelian varieties.  The case when $A$ is an abelian variety was proved by Yamanoi \cite{Yam}.

In their recent work, Levin and the first named author obtain a refinement and a new proof of Theorem \ref{tNWY} when $A=(\mathbb{C}^*)^n$ is the complex algebraic torus \cite[Theorem 1.2]{levinwang} by adapting the number-theoretical arguments of \cite{levin_gcd} to Nevanlinna theory.
Their result in this case is a complex counter parts of Theorem  \ref{LevinTheorem 1.16}.   They also prove  asymptotic gcd estimates for holomorphic maps in more general context (see  \cite[Theorem 1.8]{levinwang}). We  refer to \cite{levinwang} and \cite{PWgcd}  for further discussion on the work of gcd problems.  The proofs of our arithmetic results depend on the general version of Schmidt's subspace theorem by Ru and Vojta in \cite{ruvojta} and Levin's Subspace theorem for algebraically equivalent divisors \cite[Theorem 3.2]{levin_duke}.  The complex counter part of both results exist. (See  \cite[General theorem (analytic part)]{ruvojta} and \cite{levin_duke}.)   The formal properties of Weil functions and height functions also carry over to the corresponding functions in Nevanlinna theory.  Therefore,  the following complex result can be derived in parallel to the arithmetic case. We  refer to \cite{levinwang} for notation and basic results in Nevanlinna theory, and will omit details.

\begin{theorem}\label{cxheightEnumerical}
 Let $V$ be a   Cohen--Macaulay complex projective variety   of dimension $n$.   Let $D_1, \dots, D_{n+1} $ be    effective  Cartier divisor of $V$ in general position.   Suppose that there exists an ample Cartier divisor $A$ on $V$ and positive integer $d_i$ such that $D_i\equiv d_iA$ for all $1\le i\le n+1$.  Let $Y$ be a  closed subscheme of $V$ of codimension at least 2 that does not contain any  point  of the set $\cup_{i=1}^{n+1}( \cap_{1\le j\ne i\le n+1}\supp D_j)$.
 Let $f: \mathbb C\to V\setminus ( \cup_{i=1}^{n+1}\supp D_i)$ be a holomorphic  mapping with    Zariski-dense image.  Let $\epsilon>0$.  Then  for a given $\epsilon>0$,  we have
  \begin{align*}
T_{Y,f}(r)\leq_{\rm exc}   \epsilon T_{A,f}(r).
\end{align*}
\end{theorem}

The relevant background material will be given in the next section.  In Section \ref{twomaintheorems},  we  formulate two main theorems: one (Theorem \ref{proximityEnumerical}) for the part of the proximity function of the closed subscheme $Y$ in Theorem  \ref{heightEnumerical}, and the other (Theorem \ref{Key}) to  estimate the entire height  function of $Y$ with an extra assumption that $Y$ is not contained in any of the $D_i$.  We will deduce our main results from these two theorems.  The idea for the second theorem  is what we have described above.  The proof will be given in Section \ref{countingsec} after a quick introduction of  the work of Ru--Vojta \cite{ruvojta}.
The first theorem dealing with the proximity function part relies on Levin's generalization to a numerical equivalent class  \cite[Theorem 3.2]{levin_duke} from a version of Schimdt's subspace theorem of Evertse and Ferretti \cite[Theorem 1.6]{ef_festschrift} for a linear equivalence class of divisors.
The proof will be given in Section \ref{Proximity}.

 \section{Notation and Background Material}\label{Preliminary}
 \subsection{Heights and integral points}
We refer to  \cite[Chapter 10]{Lang}, \cite[B.8]{HinSil},  \cite[Section 2.3]{levin_gcd} or \cite[Section 2]{Sil}  for more details about this  section.
Let $k$ be a number field and $M_{k}$ be the set of places, normalized so that it satisfies the product formula
$$
\prod_{v\in M_{k}}  |x|_v=1,\qquad \text{ for }  x\in k^\times \text{.}
$$
For a point $[x_0:\cdots:x_n]\in \mathbb P^n(k)$,
the standard logarithmic height is defined by
$$
h([x_0:\cdots:x_n])=\sum_{v\in M_{k}}\log \max\{|x_0|_v,\hdots,|x_n|_v\} .
$$
This quantity is independent of the choice of the homogeneous coordinates $x_0,\hdots,x_n$ by the product formula.

We also recall that an \emph{$M_k$-constant} is a family $\{\gamma_v\}_{v\in M_k}$, where each $\gamma_v$ is a real number with all but finitely many being zero.  Given two families $\{\lambda_{1v}\}$ and $\{\lambda_{2v}\}$ of functions parametrized by $M_k$, we say $\lambda_{1v} \le \lambda_{2v}$ holds up to an $M_k$-constant if there exists an $M_k$-constant $\{\gamma_v\}$ such that the function $\lambda_{2v} - \lambda_{1v}$ has values at least $\gamma_v$ everywhere.  We say $\lambda_{1v} = \lambda_{2v}$ up to an $M_k$-constant if $\lambda_{1v} \le \lambda_{2v}$ and $\lambda_{2v} \le \lambda_{1v}$ up to $M_k$-constants.

Let $V$ be projective variety defined over a number field $k$.  The classical theory of heights associates to every Cartier divisor $D$ on $V$ a \emph{height function} $h_D:V(k)\to \mathbb R$ and a \emph{local Weil function}  (or \emph{local height function})
$\lambda_{D,v}: V(k)\setminus \supp (D)\to \mathbb R$
for each $v\in M_k$, such that
$$
\sum_{v\in M_{k}}\lambda_{D,v}(P)=h_D(P)+O(1)
$$
for all $P\in V(k)\setminus \supp (D)$.

We also recall some basic properties of local Weil functions associated to closed subschemes from \cite[Section 2]{Sil}.
Given a closed subscheme $Y$ on a projective variety $V$ defined over $k$,
we can associate to each place $v\in M_k$ a function
$$
\lambda_{Y,v}: V\setminus \supp (Y)\to \mathbb R.
$$
Intuitively, for each $P\in V$ and $v\in M_k$
$$
\lambda_{Y,v}( P)=-\log(v\text{-adic distance from $P$ to $Y$}).
$$
To describe $\lambda_{Y,v}$ more precisely, we need the following lemma:
\begin{lemma}\label{representation}
Let $Y$ be a closed subscheme of $V$.  There exist effective divisors $D_1,\hdots,D_r$ such that
$$
Y=\cap D_i.
$$
\end{lemma}
\begin{proof}
See Lemma 2.2 from \cite{Sil}.
\end{proof}

\begin{def-thm}[{\cite[Lemma 2.5.2]{vojta_lect}, \cite[Theorem 2.1 (d)(h)]{Sil}}] \label{weil}Let $k$ be a number field, and $M_k$ be the set of places on $k$.  Let $V$ be a projective variety over $k$ and let $Y\subset V$ be a closed subscheme of $V$.
We define the (local) Weil function for $Y$ with respect to $v\in M_k$ as
\begin{align}\label{WeilY}
\lambda_{Y, v}=\min_i \{\lambda_{D_i, v}\},
\end{align}
where $D_i$'s are as in the above lemma. This is independent
of the choices of the $D_i$'s up to an $M_k$-constant, and satisfies
$$
\lambda_{Y_1,v}(P) \le \lambda_{Y_2,v}(P)
$$
up to an $M_k$-constant whenever $Y_1\subseteq Y_2$.  Moreover, if $\pi:  \widetilde V \to V$ is the blowup of $V$ along $Y$ with the exceptional divisor $E$ (corresponding to the inverse image ideal sheaf of the ideal sheaf of $Y$),
$\lambda_{Y,v}(\pi(P)) = \lambda_{E, v}(P)$ up to an $M_k$-constant as functions on ${ \widetilde V}(k)\setminus E$.
\end{def-thm}

The height function for a closed subscheme $Y$ of $V$ is defined by
$$
h_Y(P):=\sum_{v\in M_k} \lambda_{Y,v}(P),
$$
for $P\in V(k)\setminus Y$.
We also define two functions related to the hight function for a closed subscheme $Y$ of $V$, depending on a finite set of places $S$ of $k$: the {\it proximity function}
$$
m_{Y,S}(P):=\sum_{v\in S} \lambda_{Y,v}(P)
$$
and the {\it counting function}
$$
N_{Y,S}(P):=\sum_{v\in M_k\setminus S} \lambda_{Y,v}(P)=h_Y(P)-m_{Y,S}(P)
$$
for $P\in V(k)\setminus Y$.
\begin{defi}
Let $k$ be a number field and $M_k$ be the set of places on $k$. Let $S\subset M_k$ be a finite subset containing
all archimedean places.  Let $X$ be a projective variety over $k$, and let $D$ be an effective divisor on $X$.  A set $R\subseteq X(k)\setminus \supp D$ is a $(D,S)$-integral set of points if there is a Weil function $\{\lambda_{D,v}\}$ for $D$ and an $M_k$-constant $\{\gamma_v\}$ such that for all $v\notin S$, $\lambda_{D,v}(P)\le \gamma_v$ for all $P\in R$.
\end{defi}
It follows directly from the definition that if $R\subseteq X(k)\setminus D$ is a $(D,S)$-integral set of points, then $N_{D,S}(P)\le O(1)$ for $P\in R$.

\subsection{Basic Propositions}
The following proposition is  an immediate consequence
of \cite[Theorem B.3.2.(f)]{HinSil}.
\begin{proposition}\label{compareheight}
Let $X$ be a projective variety defined over a number field $k$, and $A$ be an ample Cartier divisor on $X$ defined over $k$.
Let $D$ be a Cartier divisor $D$ defined over $k$ with $D\equiv A$.
Let $\epsilon>0$.
Then there exists a constant $c_{\epsilon}$ such that
$$
(1-\epsilon) h_A(P)-c_{\epsilon} \le h_D(P)\le (1+\epsilon) h_A(P)+c_{\epsilon}
$$ for   $P\in X(k)$.
\end{proposition}
\begin{proof}
Since $A$ is ample and a nonzero multiple of $A-D$ is algebraically equivalent to zero (see \cite[Remark 1.1.20]{Laza}), \cite[Theorem B.3.2.(f)]{HinSil} shows that
\[
\lim_{h_A(P) \to \infty} \frac{h_{A-D}(P)}{h_A(P)} = 0
\]
In particular, for each $\epsilon>0$, there exists $C_\epsilon$ such that when $h_A(P) \ge C_\epsilon$,
\[
-\epsilon h_{A}(P) < h_{A-D}(P) < \epsilon h_A(P),
\]
that is, $(1-\epsilon) h_A(P) < h_D(P) < (1+\epsilon) h_A(P)$.  By Northcott, there are only finitely many points $P\in X(k)$ with $h_A(P) < C_\epsilon$, we can adjust by a suitable constant $c_\epsilon$ to prove the assertion.
\end{proof}

\begin{proposition}
\label{divisorcontainY}
 Let $V$ be a  projective variety
  defined over a number field $k$ and $A$ be an ample divisor of $V$.   Let $P_1, \dots, P_{\ell} $ be  distinct points of $V(k)$.
 Let $Y$ be a closed subscheme of $V$ over $k$  which does not contain any  $P_i$, $1\le i\le \ell$.  Then there exists
 an
 effective divisor $D$  over $k$ of  $V$ linearly equivalent to a multiple of $A$ such that $Y$ is contained in
 $D$ (scheme-theoretically) and $P_i\notin D$ for any $1\le i\le \ell$.
\end{proposition}
\begin{proof}
Choose a positive integer $N$ such that  $NA$ is a very ample divisor 
of $V$.
Let $\phi:V\to \PP^m$ be the closed embedding corresponding to $NA$, i.e.
  $NA=\phi^{-1}(H) $ for some hyperplane $H$ in $\PP^m$.
 Let $I_Y $
 be the homogeneous ideals of $k[x_0,\hdots,x_m]$ defining $Y$
 as  closed subschemes of  $\PP_k^m$.
Let $F_t\in  I_Y$,
$1\le t\le r$,  be a set of homogeneous polynomials  of the same degree  that generate the ideal   $I_Y$.
Such generators exist, since any set of generators can be made the same degree by multiplying each generator with
suitable powers of $x_0, \ldots, x_m$.
Let
$$
F=a_1F_1+\dots +a_rF_r \in I_Y,
$$
where $a_1,\hdots,a_r\in k$ will be determined later.
Since $P_i\in V\setminus Y$, for $1\le i\le \ell$,
at least one of the $F_1(P_i),\dots,F_r(P_i)$ is not equal to 0.
Therefore, each of the equation
$$
 F_1(P_i)x_1+\dots + F_r(P_i)x_r=0,
$$
$1\le i\le \ell$,
determines a proper closed subvariety $H_i$ in $\PP^{r-1}$.
Choose $(a_1,\hdots,a_r)$
from $\PP^{r-1}(k)\setminus\cup_{i=1}^\ell H_i$.
Then the divisor $D=[F=0]\cap V$
contains $Y$  as a subscheme  and $\supp  D$ does not contain any $P_i$, $1\le i\le \ell$.
Moreover, the divisor $W:=[F=0]$ on $ \PP^m$ is linearly equivalent to $ (\deg F) \cdot  H$ and hence its pull-back $\phi^{-1}(W)=D$ is linearly equivalent to $ (\deg F\cdot N)A$.
\end{proof}

\section{Proof of Theorem \ref{heightEnumerical} and Theorem \ref{dim2}}\label{twomaintheorems}
Theorem \ref{heightEnumerical} and Theorem \ref{dim2} are  consequences of the following two theorems.

\begin{theorem}\label{proximityEnumerical}
 Let $V$ be a  projective variety   of dimension $n$  defined over a number field $k$, and let $S$ be a finite set of places of $k$.  Let $D_1, \dots, D_{n+1} $ be    effective  Cartier divisors of $V$   defined over $k$ in general position.   Suppose that there exist an ample Cartier divisor $A$ on $V$ and positive integers $d_i$ such that $D_i\equiv d_iA$ for all $1\le i\le n+1$.  Let $Y$ be a  closed subscheme of $V$  which does not contain any  point  of the set $\cup_{i=1}^{n+1}( \cap_{1\le j\ne i\le n+1}\supp D_j)$.  Let $\epsilon>0$.  Then there exists a proper Zariski closed set $Z$ such that for any set $R$ of $(\sum_{i=1}^{n+1} D_i,S)$-integral points on $V(k)$ we have
 \begin{align}
m_{Y,S}(P)\le   \epsilon h_A(P)+O(1)
\end{align}
for all $P\in R \setminus Z$.
\end{theorem}

\begin{theorem}\label{Key}
Under the assumptions of Theorem \ref{proximityEnumerical}, let us further assume that $V$ is Cohen--Macaulay  and that the support of $Y$ is not contained in any support of   $D_i$, $1\le i\le n+1$.
Then for any $\epsilon >0$, there exists a proper Zariski closed set $Z$ such that for any set $R$ of $(\sum_{i=1}^{n+1} D_i,S)$-integral points on $V(k)$ we have
 \begin{align}\label{Keyeq}
h_Y(P)\le   \epsilon h_A(P)+O(1)
\end{align}
for all $P\in R \setminus Z$.
\end{theorem}

\begin{proof}[Proof of Theorem  \ref{heightEnumerical}]
First, suppose that the support of $Y$ is contained in $\supp  D_i$ for some $1\le i\le n+1$.  Then $Y$ with the reduced subscheme structure is contained in $D_i$ scheme-theoretically.  Moreover, since some suitable power of the ideal of the reduced structure is contained in the ideal of $Y$, it follows from \cite[Theorem 2.1(c)]{Sil} that $h_Y$ is less than or equal to a multiple of the height with respect to $Y$ with the reduced structure.  Therefore, we may assume that $Y$ has the reduced structure and $Y\subset D_i$ scheme-theoretically.
Consequently,
the definition of the local Weil function implies that
$$
N_{Y,S}(P) \le N_{D_i,S}(P)+O(1)\le O(1),
$$
for all $P\in R$.
On the other hand, we may apply Theorem \ref{proximityEnumerical} to find  a proper Zariski closed set $Z$ such that for any set $R$ of $(\sum_{i=1}^{n+1} D_i,S)$-integral points on $V(k)$ we have
 \begin{align}
m_{Y,S}(P)\le   \epsilon h_A(P)+O(1)
\end{align}
for all $P\in R \setminus Z$.
Then the assertion is satisfied by combining the above two inequalities.

Suppose next that the support of $Y$ is not contained in any $\supp  D_i$ for $1\le i\le n+1$. This case is covered by Theorem \ref{Key}.
\end{proof}
\begin{proof}[Proof of Theorem  \ref{dim2}]
Since $Y$ is a set of finitely many points,   it suffices to show for $Y$ being a point $Q$ with a reduced subscheme structure.
If $Q\in D_i$ for some $i$, then
$$
N_{Q,S}(P) \le N_{D_i,S}(P)+O(1)\le O(1).
$$
If $Q\notin  D_i$  for each $1\le i\le n+1$, then   $Q$ is not a point in  $\cup_{i=1}^{n+1}( \cap_{1\le j\ne i\le n+1}\supp D_j)$.  Then  by Theorem \ref{Key}, there exists a proper Zariski closed set $Z_Q$ such that for any set $R$ of $(\sum_{i=1}^{n+1} D_i,S)$-integral points on $V(k)$ we have
 \begin{align}
N_{Q,S}(P) \le h_{Q}(P)\le   \epsilon h_A(P)+O(1)
\end{align}
for all $P\in R \setminus Z_Q$.
\end{proof}

\begin{proof}[Proof of Corollary \ref{CorvajaZannier}]
By replacing $f$ by $f^{\deg g}$ and $g$ by $g^{\deg f}$, we may assume that $\deg f=\deg g=d$.
Let $F(X,Y,Z)=Z^{d}f(\frac XZ,\frac YZ)$ and $G(X,Y,Z)=Z^{d}f(\frac XZ,\frac YZ)$ be the homogenization of $f(X,Y)$.  Let $\mathcal F$ and $\mathcal G$ be the projective curves defined by $F$ and $G$ respectively.  Recall the following (standard) Weil functions for $v\in M_k$:
$$
\lambda_{\mathcal F,v}(x,y,z)=-\log\frac{|F(x,y,z)|_v}{ \|f\|_v\cdot \max\{|x|^d_v,|y|^d_v,|z|^d_v\}}\ge 0,
$$
where $ \|f\|_v$ is the maximum of the $v$-adic absolute value of the  coefficients of $f$.
In particular, for $v\in M_k\setminus S$ and $x,y\in  \mathcal O_S^{\times}$, we have
$$
\lambda_{\mathcal F,v}(x,y,1)=-\log |F(x,y,1)|_v +  \log \|f\|_v = -\log |f(x,y)|_v+  \log \|f\|_v,
$$
and hence
\begin{align}
-\log^- |f(x,y)|_v\le \lambda_{\mathcal F,v}(x,y,1)-\log^- \|f\|_v
\end{align}
for $v\in M_k\setminus S$ and $x,y\in  \mathcal O_S^{\times}$.
We note that the corresponding equations hold for $g$ respectively.
Let $D_1:=\{X=0\}, D_{2}:=\{Y=0\}$, and $D_{3}:=\{Z=0\}$ be the divisors given by the coordinate hyperplanes in $\mathbb P^2$.
Then Theorem \ref{dim2} for the set $W$ of points defined as the   scheme-theoretic intersection of $\mathcal F$ and $\mathcal G$ implies that there exists a proper Zariski closed set $Z$ of $\mathbb P^2(k)$ such that
\begin{align*}
\sum_{v\in M_k\setminus S}&-\log^-\max\{ |f(x,y)|_v, |g(x,y)|_v\} \\
&\le  \sum_{v\in M_k\setminus S}\min\{\lambda_{\mathcal F,v}(x,y,1), \lambda_{\mathcal G,v}(x,y,1)\}
+O(1)\\
& = \sum_{v\in M_k\setminus S} \lambda_W(x,y,1) + O(1) \qquad \text{($\because$ Definition-Theorem \ref{weil})} \\
&\le \epsilon \max\{h(x), h(y)\}+O(1)
\end{align*}
for all $(x,y)\in \mathbb G_m^2(\mathcal O_S)\setminus Z$.
We note that the constant $O(1)$ can be eliminated by adding a finite number of points in $Z$.
Finally, since $x$ and $y$ are $S$-units, we can conclude the proof by the standard reduction argument that  the Zariski closure of a subset of $\mathbb G_m^2(\mathcal O_S)$ is a finite union of translates of algebraic subgroups (of dimension 1) by Laurent \cite{laurent}.
\end{proof}

 \section{Proof of Theorem \ref{Key}}\label{countingsec}

 \subsection{Theorem of Ru--Vojta and some basic propositions}
We first recall the following definitions and geometric properties from \cite{ruvojta}.
\begin{defi}
Let
$\mathcal L$ be a big line sheaf  and let $D$ be a nonzero effective Cartier divisor on a projective variety $X$.  We define
$$
\gamma_{\mathcal L,D } :=\limsup_{N\to\infty}\frac{N\cdot h^0(V,\mathcal L^N)}{\sum_{m=1}^{\infty} h^0 (V ,  \mathcal L^N(-m D)  ) }.
$$
\end{defi}

\begin{defi}
Let $D_1,\hdots,D_q$  be effective Cartier divisors on a variety $X$ of dimension $n$.
\begin{enumerate}
\item[{\rm (i)}] We say that $D_1,\hdots,D_q$ are {\it in general position} if for any $I\subset \{1,\dots,q\}$, we have $\dim (\cap_{i\in I} \supp  D_i) \le n-\#I $, with the convention $\dim \emptyset = -\infty$.
\item[{\rm (ii)}] We say that $D_1,\hdots,D_q$ are {\it intersect properly} if for any $I\subset \{1,\dots,q\}$ and $x\in  \cap_{i\in I} \supp  D_i$, the sequence $(\phi_i)_{i\in I}$ is a regular sequence in the local ring $\mathcal O_{X,x},$ where $\phi_i$ are the local defining equations of $D_i$, $1\le i\le q$.
\end{enumerate}
\end{defi}
\begin{remark}\label{gpCM}
If $D_1,\cdots,D_q$ intersect properly, then they are in general position.
By \cite[Theorem 17.4]{Matsumura}, the converse holds if $X$ is Cohen-Macaulay.
\end{remark}

We will use the following result to prove Lemma  \ref{Key}.

\begin{theorem} \cite[General Theorem (Arithmetic Part)]{ruvojta}\label{Ru-Vojta} Let $k$ be a number field and $M_k$ be the set of places on $k$. Let $S\subset M_k$ be a finite subset. Let $X$ be a projective variety defined over $k$.
 Let $D_1,\hdots,D_q$  be effective Cartier divisors intersecting properly on $X$.
 Let  $\mathcal L$ be a big line sheaf on $X$.  Then for any $\e>0$
\begin{align}
 \sum_{i=1}^q m_{D_i,S}(x)\le (\max_{1\le i\le q} \{\gamma_{\mathcal L, D_i} \}+\e) h_{\mathcal L}(x)
\end{align}
holds for  all $x$ outside a proper Zariski-closed subset of $X(k)$.
\end{theorem}

The following proposition is well-known in the case when $X$ and $Y$ are nonsingular (and presumably even in this setting to the experts):

\begin{proposition}\label{CohenMacaulay}
Let $X$ be a Cohen-Macaulay scheme over $k$ and $Y\subset X$ be a locally complete intersection subscheme.  Let  $\pi: \widetilde X \longrightarrow X$ be the blowup of $X$ along $Y$.  Then $ \widetilde X$ is a Cohen-Macaulay scheme.  Moreover, if $Z$ is an irreducible subscheme of $Y$,
$$
\dim \pi^{-1}(Z)=\dim Z+ \codim Y-1.
$$
\end{proposition}

\begin{proof}
Cohen-Macaulayness follows from \cite[Proposition 5.5(i)]{Kovacs}. For the second part, let $\mathscr I$ be the ideal sheaf corresponding to $Y$.  Then \cite[Proposition 5.5(ii)]{Kovacs} shows that $\mathscr I/\mathscr I^2$ is locally free, and is of rank $\codim Y$.  Since
\cite[Proposition 5.5(iii)]{Kovacs} shows that over $Y$, $\pi$ is isomorphic to the projective space bundle $\mathbb P(\mathscr I/\mathscr I^2) \to Y$, the conclusion follows.
\end{proof}

We will use the  the following  asymptotic Riemann-Roch theorem for nef divisors (see \cite[Corollary 1.4.41]{Laza}).
\begin{theorem}\label{RR}
Let $X$ be a projective variety of dimension $n$ and $D$ be a nef divisor on $X$.  Then
$$
h^0(X,\mathcal O(mD))=\frac{D^n}{n!}\cdot m^n+O(m^{n-1}).
$$
\end{theorem}
For simplicity of notation, we let $h^0(X,D):=h^0(X,\mathcal O(D))$.

 \subsection{Proof of Theorem \ref{Key}}\label{sec:proofofKey}
 We need a technical lemma.
 \begin{lemma}\label{countinglambda}
Let $V$ be a
projective variety of dimension $n$.
Let $D_1, \dots, D_{n+1} $ be    effective  Cartier divisors of $V$ defined over $k$ in general position.   Suppose that there exists an ample Cartier divisor $A$ on $V$   such that $D_i\equiv  A$ for all $1\le i\le n+1$.   Let $Y$ be a   closed subscheme of $V$ of codimension at least 2.    Let $\pi: \widetilde V\to V$ be the blowup along $Y$, and $E$ be the exceptional divisor.  Let $D:=D_1+\cdots +D_{n+1}$.  Then   for all sufficiently large
$\ell$,
  the divisor $\ell \pi^* D- E$ is ample and
\begin{align}\label{beta}
\gamma_{\mathcal L,\pi^*{D_i} }  \le \frac  1{\ell}\left(1+  O\left(\frac1{\ell^{2}}\right)\right) \le \frac 1\ell \left(1 + \frac 1{\ell\sqrt \ell}\right),
\end{align}
where $\mathcal L=\mathcal O(\ell \pi^* D- E)$.
\end{lemma}

\begin{proof}[Proof of Theorem \ref{Key}]
Since the set of $(\sum_{i=1}^{n+1} D_i,S)$-integral points do not change by replacing $D_i$ with a multiple $m_i D_i$, with $m_i=\prod_{1\le j\ne i\le n+1} d_j$, we will assume that $D_i\equiv A$ for each $1\le i\le n+1$.
 By Lemma \ref{representation}, we find an effective Cartier divisor $H_1$ which contains $Y$.  Since the codimension of $Y$ is at least two, we can choose a point inside each irreducible component of $\supp H_1$ lying outside of $\supp Y$; let $T$ be the set of such points.  Now, we apply Proposition \ref{divisorcontainY}, to obtain an effective Cartier divisor $H_2$ containing $Y$ while not containing  any point  of the set $T\bigcup \cup_{i=1}^{n+1}( \cap_{1\le j\ne i\le n+1}\supp D_j)$. Since the $\supp(H_2)$ intersected with each irreducible component of $\supp(H_1)$ is a proper subset of the irreducible component, it follows that $Y'=H_1\cap H_2$ has codimension $2$ in $V$ and $H_1$ and $H_2$ are in general position.
Consequently,  $Y'$ is locally defined by the local equations of $H_1$ and $H_2$ which form a regular sequence, as $V$ is Cohen-Macaulay (See  Remark \ref{gpCM}).
Therefore,
$Y'$ is of locally complete intersection.
Moreover, $Y'$ is  not contained in any of the $D_i$ since $Y'\supset Y$.
As $h_Y\le h_{Y'}+O(1)$, we may assume $Y$ is of locally complete intersection  of codimension exactly $2$  by replacing $Y$ with $Y'$.

Let $\pi: \widetilde V\to V$ be the blowup along $Y$, and $E$ be the exceptional divisor.
Since the support of $Y$ is not contained in any $\supp D_i$, $1\le i\le n+1$, then $\pi^* D_i$ is the strict transform $ \widetilde D_i$ of $D_i$, for each $1\le i\le n+1$.    Moreover,  $\pi |_{ \widetilde D_i} :  \widetilde D_i\to D_i$ is the blowup along $D_i\cap Y$ (See \cite[Chapter II, Corollary 7.15]{Ha});   in particular, $\pi(\widetilde {D_i}) = D_i$.

\begin{claim}
$ \widetilde D_1,\dots, \widetilde D_{n+1}$ are in general position.
\end{claim}

\begin{proof}
We first check their intersection is empty.  Suppose that $q\in  \widetilde D_i$, for each $1\le i\le n+1$.  Then $\pi(q)\in \cap_{i=1}^{n+1}\supp  D_i$ which is not possible as $D_1,\hdots,D_{n+1}$ are in general position.
Next, let $I\subset \{1,\dots,n+1\}$ with $\# I\le n$.
We show that $\dim (\cap_{i\in I} \supp  \widetilde D_i)\le n-\#I $.
By rearranging the index of the $D_i$, we let $I=\{1,\hdots,r\}$.
By assumption, $\cap_{i=1}^r D_i$ has dimension at most $n-r$, but this dimension is also at least $n-r$
because each $D_i$ is ample (see \cite[Corollary 1.2.24]{Laza}).
Now, if it has an irreducible component $W$ (of dimension $n-r$) lying inside $Y$, then
$W.D_{r+1}.\cdots .D_{n}  = W.A^{n-r} >0$, contradicting the assumption that $Y$ does not contain points in $\cap_{i=1}^n D_i$.  Therefore, $\dim (\cap_{i=1}^r D_i ) \setminus Y = \dim \cap_{i=1}^r D_i = n-r$.
Since $(\cap_{i=1}^r   \widetilde {D_i})\setminus \pi^{-1}(Y)$ and $(\cap_{i=1}^r  D_i)   \setminus Y$ are isomorphic, if we prove that any component $W$ of $(\cap_{i=1}^r   \widetilde {D_i})$ which is contained in $\pi^{-1}(Y)$ satisfies $\dim W\le n-r$, then $\dim (\cap_{i=1}^r \widetilde{D_i}) = \dim (\cap_{i=1}^r \widetilde{D_i} \setminus \pi^{-1}(Y)) = n-r$ and we are done.
Now, let $Z=\pi(W)\subset (\cap_{i=1}^r D_i )\cap Y $.  Since  we have shown that no component of $\cap_{i=1}^r D_i $ is contained in $Y$, $\dim Z\le n-r-1$.
By Proposition \ref{CohenMacaulay}, $\dim \pi^{-1}(Z)=\dim Z+1\le n-r$.
Since $W\subset\pi^{-1}(Z)$, $\dim W\le n-r$.
This completes the proof of the claim.
\end{proof}

Since $Y$ is of locally complete intersection, $ \widetilde V$ is Cohen-Macaulay by Proposition \ref{CohenMacaulay} and hence $ \widetilde D_1,\dots, \widetilde D_{n+1}$ intersect properly.
Let $D:=D_1+\cdots +D_{n+1}$.
 Let $\ell$ be a fixed positive integer satisfying $\frac{4(n+1)}{\sqrt \ell} < \epsilon$ such that the line sheaf $\mathcal L = \mathcal O(\ell \pi^* D- E)$ is ample and Lemma \ref{countinglambda} holds true, i.e.
\begin{align}
\gamma_{\mathcal L,\pi^*{D_i} }  \le \frac  1{\ell}\left(1+  \frac1{\ell \sqrt \ell}\right).
\end{align}
Let $\epsilon' = \ell^{-5/2}$.
Applying Theorem \ref{Ru-Vojta} to $\epsilon'$, $ \widetilde V$, $\mathcal L$ and  $\pi^* D_i =  \widetilde{D_i}$ (for $1\le i\le  n+1$), we have
\begin{align}
\sum_{i=1}^{n+1}m_{\pi^*D_i,S}(x) &\le  \left(\frac1\ell(1+  \frac1{\ell\sqrt \ell})+\epsilon'\right)h_{\ell \pi^* D- E}(x)\nonumber \\
&= \left(1 + \frac 2{\ell\sqrt \ell}\right) h_{\pi^* D}(x) - \left(\frac 1\ell + \frac 2{\ell^2 \sqrt \ell}\right) h_E(x) \label{eq:fromruvojta}
\end{align}
holds for  all $x$ outside a proper Zariski-closed subset $Z $ of $ \widetilde V(k)$.
By the functorial properties of the local height functions, $h_D=m_{D,S}+N_{D,S}$, and $h_E = h_Y\circ \pi$, we have
\begin{align}\label{eq:fromfunct}
\frac{1 }{\ell}h_Y(\pi(x))\le   \frac 2{\ell \sqrt \ell}\cdot h_D(\pi (x))+  \sum_{i=1}^{n+1}N_{D_i,S}(\pi(x))
\end{align}
holds for  all $\pi (x)$ outside a proper Zariski-closed subset $Z'$ of $ V(k)$.  Since $D\equiv (n+1)A$, by Proposition \ref{compareheight}, we may replace
\begin{equation}\label{eq:replaceD}
h_D(\pi (x))\le (n+1)(1+\epsilon')h_A(\pi (x))+c_1
\end{equation}
Let $R$ be a set of $(D,S)$-integral points on $V(k)$.
Then there is a constant $c_2$ such that
$$
\sum_{i=1}^{n+1}N_{D_i,S}(P)\le c_2
$$
for all $P\in R$.
Consequently, \eqref{eq:fromfunct} implies
\begin{align}\label{heightbound}
 h_Y(P) &\le   \frac{2 (n+1)(1+\epsilon')}{\sqrt \ell} \cdot h_A(P)+ \frac{ 2 c_1}{\sqrt \ell}+\ell c_2\cr
 &\le \frac{4(n+1)}{\sqrt \ell} h_A(P)+\frac{ 2 c_1}{\sqrt \ell} + \ell c_2\cr
 &< \epsilon h_A(P) + O(1)
\end{align}
holds for  all for all  $P\in R\setminus Z'$.
\end{proof}

\begin{proof}[Proof of Lemma \ref{countinglambda}]
Let $\ell$ be a fixed large integer, and let $M\ge \ell^3$ be a sufficiently large integer.
Let $\mathcal L=\mathcal O(\ell \pi^* D- E)$.
Since $D$ is numerically equivalent to $(n+1)A$, Theorem \ref{RR} implies
\begin{align}\label{total}
h^0( \widetilde V,  {\mathcal L}^{M\ell}) &= h^0( \widetilde V,  M\ell(\ell\pi^* D- E))\cr
&=\frac{(M\ell)^n(\ell\pi^* D- E)^n}{n!} + O((M\ell)^{n-1})\cr
&=\frac{\big(((n+1)\ell)^nA^n+C\big)(M\ell)^n }{n!} + O((M\ell)^{n-1}),
\end{align}
where
\begin{align}\label{C}
C=\sum_{j=1}^n\binom {n}{j}(\ell(n+1))^{n-j}(\pi^*A)^{n-j}(-E)^j.
\end{align}
Let $m\le   (n+1) M\ell^2-M $ be a positive integer.
Similarly,
\begin{align}\label{Eachm}
h^0( \widetilde V, {\mathcal L}^{M\ell}(-m\pi^*D_i)) &= h^0( \widetilde V,  M\ell(\ell\pi^* D- E)-m\pi^*D_i)\cr
&= h^0( \widetilde V,  ((n+1)M\ell^2  -m)\pi^* A- M\ell E)\cr
&\ge h^0\big( \widetilde V,  M  \big(( (n+1)\ell^2  -[\frac {m }{M }]-1 )\pi^* A-  \ell E\big)\big)\cr
&=\frac{((n+1)\ell^2 -[\frac {m }{M }]-1 )^n A^n+  C_{[\frac {m }{M }],\ell}}{n!} M^{n}+ O( M^{n-1})\cr
&\ge \frac{( (n+1)\ell^2  -\frac {m }{M}-1)^n  A^n+  C_{[\frac {m }{M }],\ell}}{n!} M^{n}+ O( M^{n-1})\cr
&=  \frac{( (n+1)M\ell^2 -M-m)^n  A^n+ C_{[\frac {m }{M }],\ell}  M^{n}}{n!} + O( M^{n-1}),
\end{align}
where
\begin{align}\label{Cm}
C_{[\frac {m }{M }],\ell}=\sum_{j=1}^n\binom {n}{j}\ell^j((n+1)\ell^2 -[\frac {m }{M }]-1 )^{n-j}(\pi^*A)^{n-j}(-E)^j.
\end{align}
  Since $0\le \frac {m }{M }\le   (n+1)  \ell^2-1$,
we let
\begin{align}\label{Cell}
C_\ell=\max_{0\le j\le (n+1)  \ell^2-1}\{|C_{j,\ell}|\}.
\end{align}
Then we have
\begin{align*}
n!\sum_{m\ge 1} &h^0( \widetilde V, {\mathcal L}^{M\ell}(-m\pi^*D_i))\cr
&\ge \sum_{m=1}^{(n+1) M\ell^2-M} \bigg(( (n+1)M\ell^2 -M-m)^n  A^n+ C_{[\frac {m }{M }],\ell}  M^{n}+ O(M^{n-1})\bigg)\cr
&\ge \frac {( (n+1)M\ell^2 -M-1)^{n+1} A^n}{n+1} - ((n+1) M\ell^2-M)C_{\ell}  M^{n}+O(M^n\ell^2)\cr
&= ((n+1)\ell  -\frac1\ell   -\frac 1{M\ell})^{n+1} \frac {A^n(M\ell)^{n+1}}{n+1} -((n+1)\ell-\frac1\ell)(M\ell)^{n+1}\frac{C_{\ell}}{\ell^n}+O(M^n\ell^2).
\end{align*}
Consequently,
\begin{align*}
\gamma_{\mathcal L,\pi^*{D_i} } &=\limsup_{M\to\infty} \frac{M\ell h^0( \widetilde V,  {\mathcal L}^{M\ell}) }{\sum_{m\ge 1}  h^0( \widetilde V, {\mathcal L}^{M\ell}(-m\pi^*D_i))} \cr
&\le\frac{((n+1)\ell)^n A^n+C}{((n+1)\ell  -\frac1\ell)^{n+1} \frac {A^n }{n+1}-((n+1)\ell  -\frac1\ell)\frac{  C_{\ell}}{\ell^n}}\cr
&=\frac  {c_{\ell}}{ \ell} ,
\end{align*}
where
$$
c_{\ell}=\frac{1+(\frac  1{(n+1) \ell})^n\frac{C}{A^n}}{ (1-\frac 1{(n+1)\ell^2} )^{n+1}- \big( \frac 1{(n+1)^{n-1}\ell^{2n}}  - \frac 1{(n+1)^n\ell^{2n+2}}\big) \frac{C_{\ell}}{A^n}}.
$$
We note that $(\pi^* A)^{n-1}E=A^{n-1}Y=0$ since the codimension of $Y$ is  at least 2.
(See \cite[Chapter VI. Proposition 2.11]{Kollar}).
  Therefore,   \eqref{C} and \eqref{Cm} satisfy
\begin{align*}
C &=\binom {n}{2}(n+1)^{n-2}(\pi^*A)^{n-2}E^2 \cdot \ell^{n-2}+O(\ell^{n-3})\\
C_{[\frac {m }{M }],\ell} &= \binom n2 (n+1)^{n-2} (\pi^*A)^{n-2} E^2 \cdot \ell^{2n-2}  + O(\ell^{2n-3}).
\end{align*}
Then for $\ell$ sufficiently large, $c_\ell = 1 + O(\frac 1{\ell^2})$, proving the claim.
\end{proof}

\section{Proof of Theorem \ref{proximityEnumerical}}\label{Proximity}

We recall the following theorem  of Levin, which is a generalization of a result of Evertse and Ferretti \cite[Theorem 1.6]{ef_festschrift}.
\begin{theorem}\cite[Theorem 3.2]{levin_duke}\label{LevinDuke}
Let $X$ be a projective variety of dimension $n$ defined over a number field $k$. Let $S$ be a finite set of places of $k$.  For each $v\in S$, let $D_{0,v},\hdots,D_{n,v}$ be effective Cartier divisors on $X$, defined over $k$, in general position.  Suppose that there exists an ample Cartier divisor $A$ on $X$ and positive integer $d_{i,v}$ such that $D_{i,v}\equiv d_{i,v}A$ for all $i$ and all $v\in S$.  Let $\epsilon>0$.  Then there exists a proper Zariski-closed subset $Z\subset X$ such that for all points $P\in X(k)\setminus Z$,
$$
\sum_{v\in S}\sum_{i=0}^n\frac{\lambda_{D_{i,v}, v}(P)}{d_{i,v}}<(n+1+\epsilon)h_A(P).
$$
\end{theorem}

\begin{lemma}\label{subgeneralweil}
Let $D_1,\cdots,D_q$  be effective divisors on a projective variety $V$  of dimension $n$, defined over $k$, in  general position.
Then
\begin{align}\label{l-inequality}
\sum_{i=1}^q \l_{D_i,v}\le  \max_I \sum_{j\in I}  \l_{D_{j},v}
\end{align}
up to an $M_k$ constant, where $I$ runs over all index subsets of $\{1,\cdots,q\}$ with $n$ elements.
\end{lemma}
\begin{proof}
Let $i_1,\cdots,i_q$ be a rearrangement of $1,\cdots,q$.
Since the $D_i, 1\leq i\leq q,$ are in  general position,
$\cap_{t=1}^{n+1}\supp D_{i_t}=\emptyset$.
Then
\begin{align}\label{mini}
\left\{\min_{1\le t\le n+1}\l_{D_{i_t},v}\right\}=\{\l_{\cap_{t=1}^{n+1}D_{i_t},v}\}
\end{align}
is bounded by an $M_k$-constant.  Therefore,
for points $P$ satisfying
\begin{align*}
 \l_{D_{i_1},v}(P)\ge  \l_{D_{i_2},v}(P)\ge\cdots\ge  \l_{D_{i_q},v}(P),
\end{align*}
we have
\begin{align*}
\sum_{i=1}^q \l_{D_i,v}(P)=  \sum_{t=1}^n \l_{D_{i_t},v}(P)
\end{align*}
up to an $M_k$-constant.
Then the assertion (\ref{l-inequality}) follows directly as the number of divisors under consideration is finite.
\end{proof}

\begin{cor}\label{proximityEnumericalDivisor}
 Let $V$ be a  projective variety of dimension $n$ defined over a number field $k$, and let $S$ be a finite set of places of $k$.  Let $D_1, \dots, D_{n+1} $ be    effective  Cartier divisors of $V$  defined over $k$  in general position.   Suppose that there exist an ample Cartier divisor $A$ on $V$ and positive integers $d_i$ such that $D_i\equiv d_iA$ for all $1\le i\le n+1$.      Let $D$ be  an Cartier divisor of $V$ defined over $k$ so that $D\equiv d_0 A$ for some positive integer $d_0$ and its support does not contain any  point  of the set $\cup_{i=1}^{n+1}( \cap_{1\le j\ne i\le n+1}\supp D_j)$.  Then for a given $\epsilon>0$, there exists a proper Zariski closed set $Z$ such that  for any set $R$ of $(\sum_{i=1}^{n+1} D_i,S)$-integral points on $V(k)$ we have
 \begin{align}\label{eq:corstatement}
m_{D,S}(P)\le   \epsilon h_A(P)+O(1)
\end{align}
for all $P\in R\setminus Z$.
\end{cor}

\begin{proof}
Let $D_0:=D$, and let $d$ be the least common multiple of $d_0,\hdots,d_{n+1}$. When we replace each $D_i$ by $d/d_i D$ and $A$ by $dA$, the set of integral points do not change and \eqref{eq:corstatement} only changes by a multiple. Therefore, we may assume that $d = d_0 = d_1 = \cdots = d_n = 1$.
We observe that  $D_0:=D$, $D_1, \dots, D_{n+1} $ are in general position   for the following reasons.
Let $I\subset \{0,1,\dots,n+1\}$.
First of all, the assumption implies that $\cap_{j\in I} \supp  D_j=\emptyset$ if $\# I> n$.
If $i:=\# I\le n,$ then $\dim (\cap_{j\in I} \supp  D_j)\ge n-i $   because each $D_j$ is ample (see \cite[Corollary 1.2.24]{Laza}).
Suppose  that $\dim (\cap_{j\in I} \supp  D_j)=n-r > n-i$.  Let $j_1,\hdots,j_{n-i+2}$ be the distinct elements of $\{0,1,\dots,n+1\}\setminus I$.
Since $D_0,\hdots,D_{n+1}$ are ample, $\dim (\cap_{j\in I} \supp  D_j)\cap (\cap_{t=1}^{n-i+1}\supp D_{j_{t}} )\ge n-r-(n-i+1)=i-r-1\ge 0$.
This is impossible since $(\cap_{j\in I} \supp  D_j)\cap (\cap_{t=1}^{n-i+1}\supp D_{j_{t}} )$ is empty.

Combining Theorem \ref{LevinDuke} and   Lemma \ref{subgeneralweil}, we find a proper Zariski-closed subset $Z\subset V$ such that for all points $P\in V(k)\setminus Z$
\begin{align}\label{proximity0}
 \sum_{i=0}^{n+1} m_{D_{i}, S}(P) <(n+1+\epsilon) h_A(P).
\end{align}
On the other hand, for a set $R$ of $(\sum_{i=1}^{n+1} D_i,S)$-integral points on $V(k)$, we have for $1\le i\le n+1$,
$$
N_{D_{i}, S}(P)=O(1)
$$
for   $P\in R$.
Therefore,
\begin{align}\label{proximity2}
\sum_{i=1}^{n+1}m_{D_{i}, S}(P) =\sum_{i=1}^{n+1}h_{D_{i}, S}(P)+O(1) \ge (n+1)(1-\epsilon) h_A(P)+O(1)
\end{align}
 for   $P\in R$, by  Proposition \ref{compareheight}.
Then by \eqref{proximity0} and \eqref{proximity2}, we have
\begin{align}\label{proximity1}
m_{D,S}(P)=m_{D_0,S}(P)< (n+2) \epsilon   h_A(P)+O(1),
\end{align}
for $P\in R\setminus Z$.
\end{proof}

\begin{proof}[Proof of Theorem \ref{proximityEnumerical}]
 By Proposition \ref{divisorcontainY},  there exists an effective divisor $D$ linearly equivalent to  $dA$ for some positive integer $d$ with the property that $D\supset Y$ and its support does not contain any point of $\cup_{i=1}^{n+1}( \cap_{1\le j\ne i\le n+1}\supp D_j)$.  Then we arrive immediately that
$$
m_{Y,S}(P)\le m_{D,S}(P)+O(1)
$$ for any $P\notin \supp Y$, and by Corollary \ref{proximityEnumericalDivisor}
there exists a proper Zariski closed set $Z$ such that for a set $R$ of $(\sum_{i=1}^{n+1} D_i,S)$-integral points on $V(k)$
 \begin{align*}
m_{D,S}(P)\le   \epsilon h_A(P) +O(1)
\end{align*}
for all $P\in R\setminus Z$.
Hence, the assertion of the theorem is proved.
 \end{proof}

\begin{remark}\label{rem:proximity}
As in the introduction, Corollary \ref{proximityEnumericalDivisor} implies \cite[Theorem 4.1]{levin_gcd}, namely that
 \begin{align}
m_{D,S}(P)\le   \epsilon h(P) +O(1)
\end{align}
for all $P\in  \mathbb G_m^n(\mathcal O_S)$ outside a Zariski-closed proper set, as long as $D$ does not contain
$$
 P_0:=[1:0:\cdots:0],\hdots,P_n:=[0:0\cdots:0:1].
 $$
In fact, we can replace $D$ and extend to the case of a closed subscheme $Y$ on $\mathbb P^n$ not containing $P_0, \ldots, P_n$, by an argument similar to the beginning of the proof of Theorem \ref{proximityEnumerical} (see also \cite[Remark 4.2]{levin_gcd}).
\end{remark}
 \noindent\textbf{Acknowledgements.}  The first named  author
 wishes to thank the Department of Mathematics, College of Science and Technology, Nihon University for kind hospitality during which part of
the work on this paper took place.



\begin{thebibliography}{x}
\bibitem{Bugeaud:Corvaja:Zannier:2003}
\textsc{Y,~Bugeaud, P.~Corvaja, and U.~Zannier}, \emph{An upper bound for the G.C.D. of $a^n-1$ and $b^n-1$}, Math.  Z. \textbf{243} (2003), no.~1, 79--84.
 \bibitem{Corvaja:Zannier:2003}
\textsc{P.~Corvaja, and U.~Zannier}, \emph{On the greatest prime factor of $(ab+1)(ac+1)$}, Proc. Amer. Math. Soc \textbf{131} (2003), no.~6, 1705--1709.
 \bibitem{Corvaja:Zannier:2005}
\textsc{P.~Corvaja, and U.~Zannier}, \emph{A lower bound for the height of a rational function at $S$-unit points}, Monatsh. Math.  \textbf{144} (2005), no.~3, 203--224.
\bibitem{ef_festschrift}
\textsc{J.-H. Evertse and R. G.~Ferretti},
\textit{ A generalization of the {S}ubspace {T}heorem with polynomials of
  higher degree},
 In {\em Diophantine approximation}, volume~16 of  Dev. Math.,
  pages 175--198. SpringerWienNewYork, Vienna, 2008.
\bibitem{GuoWang}
\textsc{J.~Guo and J.~T.-Y.~Wang,}
\textit{Asymptotic gcd and divisible sequences for entire functions,} Transactions of the American Mathematical Society, \textbf{371} (2019), no.~9, 6241--6256.
\bibitem{Ha} \textsc{R.~Hartshorne,}
\emph{Algebraic Geometry,} Grad. Texts in Math. vol. 52,
Springer-Verlag, New York, 1997.
\bibitem{Hernandez:Luca:2003}
\textsc{S.~Hern\'andez and F.~Luca}, \emph{On the largest prime factor of $(ab+1)(ac+1)(bc+1)$}, Bol. Soc. Math. Mexicana (3)  \textbf{9} (2003), no.~2, 235--244.
\bibitem{HinSil} \textsc{M.~Hindry and J.~H.~Silverman,} \emph{Diophantine Geometry-an introduction,}
Springer-Verlag, New York, 2000.
\bibitem{Kollar} \textsc{J.~Koll\'ar,}  \emph{Rational curves on algebraic varieties,}
 Ergebnisse der Mathematik und ihrer Grenzgebiete. 3. Folge. A Series of Modern Surveys in Mathematics, \textbf{32},
Springer-Verlag, Berlin, 1996.
\bibitem{Kovacs}
\textsc{S.~J.~Kov\'acs},  \emph{Rational singularities},  arXiv:1703.02269v6.
\bibitem{Lang} \textsc{S.~Lang,} \emph{Fundamentals of Diophantine Geometry,}
Springer-Verlag, New York, 1983.

\bibitem{laurent} \textsc{M.~Laurent,} \emph{\'{E}quations diophantiennes exponentielles}, Invent. Math. {\bf 78} (1984), 299--327.

\bibitem{Laza} \textsc{R.~Lazarsfeld,} \emph{Positivity in Algebraic Geometry. I. Classical setting: line bundles and linear series.}
Ergebnisse der Mathematik und ihrer Grenzgebiete. 3. Folge. A Series of Modern Surveys in Mathematics, 48. Springer-Verlag, Berlin, 2004.
\bibitem{levin_duke}
\textsc{A.~Levin},
 \emph{On the Schmidt subspace theorem for algebraic points}, Duke Math. J. {\bf 163} (2014), no. 15, 2841--2885.
\bibitem{levin_gcd}
\textsc{A.~Levin},
 \emph{Great common divisors and Vojta's conjecture for blowup of algebraic tori},
Invent. Math. {\bf 215} (2019), no. 2, 493--533.
 \bibitem{levinwang}
\textsc{A.~Levin and J. T.-Y. Wang,}
 \textit{Greatest common divisors of analytic functions and Nevanlinna theory on algebraic tori}, arXiv:1903.03876.

\bibitem{mck}
\textsc{D.~McKinnon}, \emph{Vojta's main conjecture for blowup surfaces}, Proc. Amer.
  Math. Soc. \textbf{131} (2003), no.~1, 1--12.
\bibitem{Matsumura}\textsc{H.~Matsumura},  \emph{Commutative Ring Theory.} Translated by M. Reid. Cambridge
University Press, Cambridge, 1986. Cambridge Studies in Advanced Mathematics, \textbf{8}.
\bibitem{NWbook}
\textsc{J.~Noguchi and J.~Winkelmann,}
 \emph{Nevalinna Theory in Several Variables and Diophatine Approximation,} Grundlehren der Mathematischen Wissenschaften {\bf 350}. Springer-Verlag, Tokyo, 2014.

\bibitem{NWY}
\textsc{J. Noguchi, J. Winkelmann and K. Yamanoi,}
\emph{The second main theorem for holomorphic curves into semiabelian varieties II,} Forum Mathematicum {\bf 20} (2002), no. 3,  469--503.
\bibitem {PWgcd}
\textsc{H.~Pasten and J.~T.-Y.~Wang,}
\emph{GCD bounds for analytic functions,} International Mathematics Research Notices {\bf 160} (2017), no. 1,  47--95.
\bibitem{ruvojta}
 \textsc{M.~Ru and P.~Vojta},
\textit{Birational Nevanlinna constant and its consequences},
Amer. J. Math.,  to appear.
\bibitem{Sil} \textsc{J.~H.~Silverman,} \emph{Arithmetic distance functions and height functions in Diophantine geometry,}
 Math. Ann. {\bf 279} (1987), no. 2, 193--216.

\bibitem{Silverman2005}
\textsc{J.~H.~Silverman,} \emph{Generalized greatest common divisors, divisibility sequences, and Vojta's conjecture for blowups}, Monatsh. Math. {\bf 145} (2005), 333-350.

\bibitem{vojta_lect}
 \textsc{P.~Vojta},
 \emph{Diophantine Approximation and Value Distribution Theory},
 Volume 1239 of Lecture Notes in Mathematics, Springer-Verlag, Berlin, 1987.
\bibitem{Vojta}
\textsc{P.~Vojta,} \emph{Diophantine approximation and Nevanlinna theory}, {\it Arithmetic geometry,} 111-224, Lecture Notes in Mathematics {\bf 2009}, Springer-Verlag, Berlin, 2011.
\bibitem{Yam}
\textsc{K.~Yamanoi,}
\emph{Holomorphic curves in abelian varieties and intersections with higher codimensional subvarieties,} Forum Math. {\bf 16} (2004), no. 5,  749--78
\bibitem{yasumona}
\textsc{Y.~Yasufuku}, \emph{Vojta's conjecture on blowups of $\mathbb{P}^n$, greatest
  common divisors, and the $abc$ conjecture}, Monatsh. Math. \textbf{163}
  (2011), no.~2, 237--247.
\bibitem{yasuams}
\bysame, \emph{Integral points and {V}ojta's conjecture on rational surfaces},
  Trans. Amer. Math. Soc. \textbf{364} (2012), no.~2, 767--784.
\bibitem{yasuforum}
\bysame, \emph{Vojta's conjecture on rational surfaces and
the $abc$ conjecture}, Forum Math.  \textbf{30} (2018), no.~3, 631-649.
\end{thebibliography}
\end{document}